\documentclass{amsart}
\usepackage{config}
\usepackage{tcolorbox}

        	\definecolor{jazzberryjam}{rgb}{0.65, 0.04, 0.37}

\newtheorem{thmA}{Theorem}

\author[F. J. Cruz-Zamorano]{Francisco J. Cruz-Zamorano}
\address[F. J.  Cruz-Zamorano]{Departamento de An\'alisis Matem\'atico e IMAULL,  Universidad de La Laguna, Avenida Astrof\'isico Francisco S\'anchez, s/n, 38206 San Crist\'obal de La Laguna, Santa Cruz de Tenerife,  Spain}
\email{fcruzzam@ull.edu.es}

\author[E. Dellepiane]{Eugenio Dellepiane}
\address[E. Dellepiane]{University of Twente, Zilverling (building no. 11), Hallenweg 19, 7522 NH Enschede, Netherlands} \email{e.dellepiane@utwente.nl; dellepianeeugenio@gmail.com}

\author[D. Seco]{Daniel Seco}
\address[D. Seco]{Departamento de An\'alisis Matem\'atico e IMAULL,  Universidad de La Laguna, Avenida Astrof\'isico Francisco S\'anchez, s/n, 38206 San Crist\'obal de La Laguna, Santa Cruz de Tenerife,  Spain}
\email{dsecofor@ull.edu.es}

\thanks{Cruz-Zamorano was supported by Grant PID2025-167700NB-I00 funded by MICIU/AEI/10.13039/501100011033 and by FEDER, UE, through the Generaci\'on de Conocimiento programme of Spanish Ministry of Science, Innovation and Universities. Dellepiane is member of Gruppo Nazionale per l'Analisi Matematica, la Probabilità e le loro Applicazioni (GNAMPA) of Istituto Nazionale di Alta Matematica (INdAM).  Seco was supported by Grant PID2024-160185NB-I00 funded by MICIU/AEI/10.13039/501100011033 and by FEDER, UE, through the Generaci\'on de Conocimiento programme of Spanish Ministry of Science, Innovation and Universities; and by grant RYC2021-034744-I of the Ram\'on y Cajal programme from Agencia Estatal de Investigaci\'on (Spanish Ministry of Science, Innovation and Universities).}


\title[Compact and H--S composition operators on w. Hardy spaces]{Compact and Hilbert--Schmidt composition operators on weighted Hardy spaces}
\date{\today}

\begin{document}
\begin{abstract}
Lef\`evre, Li, Queff\'elec, and Rodr\'iguez-Piazza \cite{LLQRP} recently characterized the sequences $\beta$ for which the weighted Hardy space $H^2(\beta)$ has the same bounded composition operators as the Hardy space $H^2$. We investigate the analogous questions for the compactness and the $p$-Schatten classes: for which sequences $\beta$ does the space $H^2(\beta)$ have the same compact or $p$-Schatten class composition operators as $H^2$?

\end{abstract}

\maketitle

\section{Introduction}
Let $\beta = (\beta_n)_{n \in \N}$ be a sequence of positive numbers (also called a weight). The weighted Hardy space $H^2(\beta)$ consists of those analytic functions $f$ for which
\[
\norm{f}_{\beta}^2 := \sum_{n = 0}^{\infty}\beta_n|\widehat{f}(n)|^2 < + \infty, \qquad f(z) = \sum_{n = 0}^{\infty}\widehat{f}(n)z^n,\,\,z\in\mathbb{D}.
\]
This class of weighted Hardy spaces was considered, for instance, by Katsnel'son \cite{Katsnelson} and Shields \cite{Shields}, and it contains many important spaces of analytic functions. Among the most remarkable, when $\beta=\mathbf{1}$ is the sequence that is constantly $1$, then $H^2 := H^2(\mathbf{1})$ is the classical Hardy space. Taking $\beta_n = (n+1)^{\alpha}$, one obtains the weighted Dirichlet spaces $\mathcal{D}_\alpha$. Finally, if $\beta_n = \binom{\gamma + n+1}{n}^{-1}$, one obtains the weighted Bergman spaces $A^2_\gamma$.

It is known \cite[p. 1921]{LLQRP} that every function in $H^2(\beta)$ is holomorphic on the open unit disk $\D = \{z \in \C : \abs{z} < 1\}$ if and only if the weight $\beta$ satisfies
\begin{equation}
\label{eq:cond-D}
\liminf_{n \to \infty} \beta_n^{1/n} \geq 1.
\end{equation}
Our main aim is to study composition operators with symbols in $\operatorname{Hol}(\mathbb{D},\mathbb{D})$ acting on these spaces, so we will always assume that \eqref{eq:cond-D} holds.

Recall that, given a holomorphic self-map of the unit disk $\varphi\colon\mathbb{D}\to\mathbb{D}$, the composition operator $C_\varphi \colon \operatorname{Hol}(\mathbb D, \C)\to\operatorname{Hol}(\mathbb D,\C)$ is given by
$$C_{\varphi}(f) = f \circ \varphi, \qquad f \in \operatorname{Hol}(\mathbb D,\C),$$
where $\operatorname{Hol}(\mathbb D,\C)$ denotes the spaces of all holomorphic maps on $\D$.

The properties of these operators have been extensively studied on various spaces (we refer to \cite{CowenMacCluer,Shapiro-Book} for detailed information). They are particularly well-studied on the Hardy space $H^2$. For instance, as a consequence of the classical Littlewood subordination principle, it is known that all composition operators $C_\varphi$ act boundedly on $H^2$ \cite[p. 16]{Shapiro-Book}. This feature of the Hardy space naturally raises the following question: for which other weights $\beta$ does it hold that every symbol $\varphi\colon\mathbb{D}\to\mathbb{D}$ induces a bounded operator $C_\varphi\colon H^2(\beta)\to H^2(\beta)$? This question was originally raised by Zorboska \cite[Problem 1 in p. 49]{Zorboska}. In the recent paper \cite{LLQRP}, Lef\`evre, Li, Queff\'elec, and Rodr\'iguez-Piazza answered this problem with the following result.

\begin{thmA}
\label{T:rodriguezPiazza}
{\rm \cite[Theorem 1.2]{LLQRP}}
Let $\beta=(\beta_n)_{n \in \N}$ be a weight. Then, the composition operator $C_\varphi\colon H^2(\beta) \to H^2(\beta)$ is bounded for every symbol $\varphi\colon\mathbb{D}\to\mathbb{D}$ if and only if $\beta$ satisfies the following properties:
\begin{enumerate}[\hspace{0.5cm}\rm(i)]
\item $\beta$ is \emph{essentially decreasing}, that is, there exists a constant $C\geq 1$ such that for all $m\geq n\geq 0$,
$$\beta_m\leq C\beta_n;$$
\item $\beta$ is \emph{slowly oscillating}, that is, there exist positive constants $c<1<C$ such that
$$c\leq \frac{\beta_m}{\beta_n}\leq C \qquad \text{when}\quad 0 \leq \frac n2 \leq m \leq 2n.$$
\end{enumerate}
\end{thmA}

Introducing the set
\[
\operatorname{B}(\beta)=\{\varphi\colon\mathbb{D}\to\mathbb{D} \,\,|\,\, C_\varphi\colon H^2(\beta)\to H^2(\beta)\,\,\text{is bounded}
\},
\]
we can reformulate Theorem \ref{T:rodriguezPiazza} as follows: the identity $\operatorname{B}(\beta)=\operatorname{B}(\mathbf{1})$ holds if and only if $\beta$ is essentially decreasing and slowly oscillating. This point of view leads us to pose the following general question.

\textbf{Question.} For a given operator-theoretic property, when does the class of composition operators on $H^2$ with such a property coincide with the analogous class on $H^2(\beta)$?

In this paper we investigate this question for the case of the compactness and the Schatten classes of composition operators on $H^2(\beta)$. Many references deal with these properties on several function spaces. For instance, the characterization of compact composition operators on $H^2$ is a celebrated result by Shapiro \cite{Shapiro}, who used the so-called Nevanlinna counting function. His ideas have been greatly studied by other authors. For instance, Kellay and Lef\`evre \cite{KL-Compact} generalized Shapiro's result to weighted Bergman spaces $A^2_w$ defined by radial weights, which are a proper subclass of the spaces $H^2(\beta)$ (see \cite[p. 1922]{LLQRP}). Luecking and Zhu \cite{LueckingZhu} also obtained an analogous characterization of the $p$-Schatten classes of composition operators on~$H^2$.

Before presenting our main results, let us introduce some notation. Given a weight $\beta$ and $1\leq p<\infty$, we consider
\begin{align*}
    \operatorname{K}(\beta)&=\{\varphi\colon\mathbb{D}\to\mathbb{D} \,\,|\,\, C_\varphi\colon H^2(\beta)\to H^2(\beta)\,\,\text{is compact}
\},\\
\operatorname{S}_p(\beta) &=\{\varphi\colon\mathbb{D}\to\mathbb{D} \,\,|\,\, C_\varphi\colon H^2(\beta)\to H^2(\beta)\,\,\text{is in the}\,\,p\text{-Schatten class}
\}.
\end{align*}
For $p=2$, we also write $\operatorname{HS}(\beta)=\operatorname{S}_2(\beta)$ for the set of symbols that induce Hilbert--Schmidt composition operators on $H^2(\beta)$.  As hinted before, this article is devoted to the following questions:

\begin{itemize}
\item[(\textbf{Q1})] For which weights $\beta$ does the equality $\operatorname{K}(\beta)=\operatorname{K}(\mathbf{1})$ hold?
\item[(\textbf{Q2})] Fix $1\leq p<\infty$. For which weights $\beta$ does the equality $\operatorname{S}_p(\beta)=\operatorname{S}_p(\mathbf{1})$ hold?
\end{itemize}

We now present our main findings. First of all, we remark that in order to have either one of the desired equalities, $\beta$ is necessarily bounded (we will show this in Proposition \ref{P:betabounded}). Under the extra assumption that $\beta$ is equivalent to a weight which converges to some limit, a very rigid phenomenon occurs. This is the case, for instance, if $\beta$ is essentially decreasing \cite[p. 1925]{LLQRP}. We are ready to state our first main result. We say that $\beta \simeq \beta'$ if there exist $c,C > 0$ such that $c \beta'_n \leq \beta_n \leq C\beta'_n$ for every $n\in\mathbb{N}$. In such a case, $H^2(\beta) = H^2(\beta')$ with equivalent norms.

\begin{theorem}
\label{T:main1}
Consider a weight $\beta$. Suppose that $\beta\simeq\beta'$, where $\beta'$ is a weight for which $\lim_{n \to \infty}\beta'_n$ exists (e.g., $\beta$ is essentially decreasing). Then, the following are equivalent: 
\begin{enumerate}[\hspace{0.5cm}\rm(a)]
\item $\operatorname{K}(\beta)=\operatorname{K}(\mathbf{1})$. 
\item $\operatorname{S}_p(\beta)=\operatorname{S}_p(\mathbf{1})$ for some (hence, all) $p \geq 1$.
\item $\beta\simeq \mathbf{1}$.
\end{enumerate}
\end{theorem}

Of course, our result does not fully answer the previously stated questions. In fact, obtaining the full answer seems difficult. For the boundedness, Theorem~\ref{T:rodriguezPiazza} already reveals that full characterizations may depend on fine properties concerning the growth and the oscillation of the weight $\beta$. However, it is perhaps surprising that the compactness and the Schatten class properties of different spaces are not distinguished under a mild assumption on the weight, such as being essentially decreasing. In particular, this observation leads to a characterization of the classical Hardy space $H^2$ among all the weighted spaces $H^2(\beta)$.

\begin{corollary}
\label{cor:h2}
Let $\beta$ be a weight. Assume that $\operatorname{B}(\beta) = \operatorname{B}(\mathbf{1})$ and $\operatorname{K}(\beta) = \operatorname{K}(\mathbf{1})$. Then, $\beta \simeq \mathbf{1}$. In particular, $H^2(\beta) = H^2$ with equivalent norms.
\end{corollary}

One may wonder how strong is the convergence assumption in Theorem \ref{T:main1}. In other words, can we find examples of some weights $\beta$ with the property that
\begin{equation}
\label{eq:zero-liminf-limsup-infty}
0=\liminf_{n \to \infty}\beta_n<\limsup_{n \to \infty}\beta_n < +\infty
\end{equation}
for which $\operatorname{K}(\beta) = \operatorname{K}(\mathbf{1})$ or $\operatorname{S}_p(\beta) = \operatorname{S}_p(\mathbf{1})$ for some $1 \leq p < \infty$? To this end, we also explore weights $\beta$ with a lacunary behavior in the following sense. 

Consider a sequence of natural numbers $\Lambda=\{N_k\}_{k \in \N}$ with the property that there exists $q > 1$ such that $N_{k+1} \geq q N_k$ for all $k\in \N$. For such a set $\Lambda$, consider also the weights $\beta$ such that
\begin{equation}
\label{E:beta_lacunary}
0 < \beta_n \leq 1 \text{ for all } n\in\mathbb{N}, \qquad
\inf_{n \not\in \Lambda}\beta_n > 0.
\end{equation}
The assumption $\beta_n \leq 1$ is a mild normalization, used to ease the exposition of the proofs.
On the lacunary part, a priori we allow $\beta$ to tend to $0$. For these weights, we can completely characterize when $\operatorname{K}(\beta)=\operatorname{K}(\mathbf{1})$, and when $\operatorname{HS}(\beta)=\operatorname{HS}(\mathbf{1}).$ Notice that we are only considering the $p$-Schatten class with $p=2$. From now on we denote by $\|\varphi\|_{\mathrm{HS}(\beta)}$ the Hilbert-Schmidt norm of the composition operator $C_\varphi$ acting on $H^2(\beta)$.

\begin{theorem}
\label{thm:counterexample}
Let $q > 1$ and let $\Lambda$ be a lacunary sequence of natural numbers as above. Let $\beta$ be a weight as in \eqref{E:beta_lacunary}. The following results hold.  
\begin{enumerate}[\hspace{0.5cm}\rm(a)]
\item $\mathrm{K}(\beta) = \mathrm{K}(\mathbf{1})$ if and only if $\beta \simeq \mathbf{1}$.
\item Assume that $q>2$. Then, $\mathrm{HS}(\beta) = \mathrm{HS}(\mathbf{1})$ if and only if $\inf_{k \in \N}\beta_{N_k}N_k>0$. In such a case, there exists a constant $C>0$ such that for every $\varphi\colon\mathbb{D}\to\mathbb{D}$
\[
\frac1C \|\varphi\|_{\operatorname{HS}(\mathbf{1})} \leq \|\varphi\|_{\operatorname{HS}(\beta)}\leq C\|\varphi\|_{\operatorname{HS}(\mathbf{1})}.
\]
\end{enumerate}
\end{theorem}

This theorem shows that, as expected, the identities $\operatorname{K}(\beta)=\operatorname{K}(\mathbf{1})$ and $\operatorname{HS}(\beta)=\operatorname{HS}(\mathbf{1})$ are not equivalent in general. It also shows that a complete characterization of the equality $\operatorname{HS}(\beta) = \operatorname{HS}(\mathbf{1})$ has to take into account both the distribution of the masses (in this example, the lacunarity) and the decay of $\beta$. However, it does not give any further insights on compact composition operators. Namely, we still lack examples where $\operatorname{K}(\beta) = \operatorname{K}(\mathbf{1})$ for which \eqref{eq:zero-liminf-limsup-infty} holds.

The rest of the paper is organized as follows. We first gather some preliminary results in Section \ref{sec:preliminaries}. After that, we prove Theorem \ref{T:main1} in Section \ref{sec:proof1}. The main tool behind the proof will be the construction of a Blaschke product whose modulus goes to $1$ on the boundary of $\D$ at a controlled rate. This follows an idea of Jones \cite{Jones}. We continue with Section \ref{sec:proof2}, where we prove Theorem \ref{thm:counterexample}. Among other constructions, the main tool there will be a projection theorem due to Paley \cite{Paley}. In Section \ref{sec:further} we will examine further comparisons between $\operatorname{HS}(\beta)$ for different weights $\beta$. Finally, we end with some open questions in Section \ref{sec:questions}.

\textbf{Acknowledgements}. The authors would like to thank V\'ictor J. Maci\`a for our discussions during early stages of this work.

\section{Preliminaries}
\label{sec:preliminaries}
In this section we recall some of the main definitions and tools that will be used. 

\subsection{Hardy spaces}
For $1 \leq p <\infty,$ the Hardy space is
\[
H^p=\{f\in\operatorname{Hol}(\mathbb{D},\C)\colon \|f\|_p^p:=\sup_{0<r<1}\frac{1}{2\pi}\int_{-\pi}^{\pi}|f(re^{i\theta})|^p\operatorname{d}\!\theta<\infty \}.
\]
These are classical Banach spaces of analytic functions, of great interest for several branches of analysis, such as complex analysis, operator theory and harmonic analysis. For $p=\infty,$ we consider
\[
H^\infty = \{f\in\operatorname{Hol}(\mathbb{D},\C)\colon \|f\|_\infty =\sup_{z \in \mathbb D}|f(z)|<\infty\}.
\]
We refer to \cite{Duren-Hp} for more information about these spaces.

Hardy spaces possess many fundamental properties. For instance, it is known that $H^p \subseteq H^q$ whenever $p \geq q$. Therefore, other Hardy spaces $H^p$ inherit the properties of functions in $H^1$. One such property is a well-known consequence of a theorem of Fatou \cite[Theorem 2.2]{Duren-Hp}: every $f$ in $H^1$ admits boundary values, meaning that the radial limit $f(e^{i\theta}):=\lim_{r \to 1}f(re^{i\theta})$ exists for almost every $\theta\in (-\pi,\pi]$. 

Another important feature of the Hardy spaces is the so-called inner-outer factorization: any $f$ in $H^1$ can be written as the product
\[
f(z)= u(z) \exp\left(\dfrac{1}{2\pi}\int_{-\pi}^{\pi}\dfrac{e^{i\theta}+z}{e^{i\theta}-z}\,\log|f(e^{i\theta})|\operatorname{d}\!\theta\right), \qquad z \in \D,
\]
where $u$ is an inner function, that is, an element of $H^\infty$ that satisfies $|u|=1$ a.e. on $\partial\mathbb{D}$. Up to unimodular constants, inner functions consist of Blaschke products and singular inner functions. Namely, 
\[
u(z)=e^{i\gamma} \exp\left(\dfrac{1}{2\pi}\int_{-\pi}^{\pi}\dfrac{e^{i\theta}+z}{e^{i\theta}-z}\operatorname{d}\!\tau(\theta)\right)\prod_{n=1}^{N}B_{a_n}(z), \qquad z\in\mathbb{D},
\]
where $\gamma\in\mathbb{R}$, $\tau$ is a singular measure on the unit circle, $N \in \N \cup \{\infty\}$, $(a_n)_{1 \leq n \leq N}$ is a sequence in~$\mathbb{D}$ that satisfies $\sum_{n\in\mathbb{N}}(1-|a_n|)<\infty$, and 
\[
B_a(z) = \begin{cases}
\dfrac{\overline{a}}{|a|}\dfrac{a-z}{1-\overline{a}z}, \qquad a \in \D \setminus \{0\},  \\
z, \qquad a = 0,
\end{cases}
\qquad z \in \D.
\]
The product $B=\prod_{n = 0}^N B_{a_n}$ is called a Blaschke product. Notice that $(a_n)_n$ is precisely the sequence of zeroes of $B$, counted with multiplicity. For more information, see \cite[Section 2.4]{Duren-Hp}.

\subsection{Blaschke products}
In our proofs, we will use inner functions and the inner-outer factorization of functions in $H^2$ several times. To ease the exposition, we include here a preliminary result on Blaschke products, which should be known to experts. Recall that a Blaschke product $B$ is an inner function, meaning that
\[
\lim_{r \to 1^-}|B(r\xi)| = 1, \qquad \text{for a.e. } \xi \in \partial\D.
\]
However, such a convergence may be arbitrarily slow. This has been analyzed in \cite[Theorem 3.1]{LLQRP-Blaschke}, for instance. We will use a variant of a simpler statement due to M. M. Jones; see \cite[Theorem 10]{Jones}. 
Essentially, the only difference with the statement of Jones is the use of increasing/non-decreasing functions. We include the details for the sake of completeness.

\begin{lemma}
\label{lemma:Jones}
Let $r_0 \in (0,1)$ and let $w \colon [r_0,1] \to (0,1]$ be a continuous non-decreasing function with $w(1) = 1$. There exists a Blaschke product $B$ such that
$$M(r,B) := \sup_{\xi \in \partial\D} |B(r\xi)| < w(r), \qquad r \in (r_0,1).$$
\begin{proof}
Choose a decreasing sequence $(\varepsilon_n)_{n \in \N}$ with $\varepsilon_n \in (0,\tfrac14)$ and $\sum_{n \in \N}\varepsilon_n < \infty$. Using the properties of $w$, we can find an increasing sequence $(r_n)_{n \in \N}$ with $r_n \in (r_0,1)$ converging to $1$ and such that $w(r) \geq 1-\varepsilon_n$ for all $r_n < r < 1$.  For each $n \in \N$, we also choose $N_n \in \N$ so that $r_{n+1}^{N_n} < \varepsilon_n$.

For brevity, put $a_n = 1-2\varepsilon_n \in (0,1)$, $n \in \N$. Let us consider the Blaschke products $B_n$, $n \in \N$, given by
$$B_n(z) = e^{i\theta_n}\dfrac{z^{N_n}+a_n}{1+a_nz^{N_n}}, \qquad z \in \D,$$ where $\theta_n$ is a suitable angle.
We may consider the infinite Blaschke product $B = \prod_{n \in \N}B_n$, since the Blaschke condition is satisfied: using the elementary inequality
\[
1-e^{-x}\leq x, \qquad x\geq 0,
\]
we have that
\[
\sum_{n \in \N}N_n(1-a_n^{1/N_n}) \leq -\sum_{n \in \N}\log(a_n) = -\sum_{n \in \N}\log(1-2\varepsilon_n)< +\infty.
\]
The series converges because, by assumption, $\sum_{n\in\mathbb{N}}\varepsilon_n <\infty$.

Now, let $r \in (r_0,1)$. Find $n \in \N$ so that $r_n \leq r \leq r_{n+1}$.
A standard argument using the fact that $a_n\in(0,1)$ shows that
\[
M(r,B_n) = \dfrac{r^{N_n}+a_n}{1+a_nr^{N_n}}. 
\]
Then, by construction, we have that
\[
M(r,B) \leq M(r,B_n) \leq r^{N_n}+a_n \leq r_{n+1}^{N_n}+a_n \leq 1-\varepsilon_n \leq w(r).
\]
\end{proof}
\end{lemma}

\subsection{Weighted Hardy spaces}
Given a function $f \colon \mathbb{D} \to \mathbb{C}$ that is analytic in $0$, we denote its $n$-th Taylor coefficient by $\widehat{f}(n)$. The Hilbert space $H^2$ admits the following beautiful norm identity: 
\[
\|f\|_2^2=\sum_{n=0}^\infty |\widehat{f}(n)|^2, \qquad f\in H^2.
\]
In light of this property, it is natural to consider the following generalization: given a sequence of positive numbers $\beta = (\beta_n)_{n \in \N}$, the weighted Hardy space $H^2(\beta)$ is defined as
\[
H^2(\beta)=\{
f(z)=\sum_{n\in\mathbb{N}} \widehat{f}(n)
z^n\colon \norm{f}_{\beta}^2 := \sum_{n = 0}^{\infty}\beta_n|\widehat{f}(n)|^2 < + \infty
\}.
\]
It is important that $\beta_n>0$ for every $n\in\mathbb{N}$, in order for $\|\cdot\|_\beta$ to be a norm. In this case, $H^2(\beta)$ is a Hilbert space.

These spaces appeared in \cite{Shields}, where Shields studied the weighted shift operators. We should remark that every Hilbert space of functions which are analytic around~$0$ for which the monomials are orthogonal and polynomials are dense can be seen as a weighted Hardy space $H^2(\beta)$ for some appropriate weight $\beta$. This is the case of a number of spaces which often appear in the literature (see \cite[p. 1921]{LLQRP}), including all those mentioned in the Introduction.

In general, the functions in $H^2(\beta)$ are not necessarily holomorphic in $\D$. This will be a key property for us, since it is crucial to define composition operators with symbols in $\D$. However, it is known that \(H^2(\beta)\) is contained in $\operatorname{Hol}(\mathbb{D},\mathbb{C})$ if and only if
$$
\liminf_{n \to \infty} \beta_n^{1/n} \geq 1,
$$
see for example \cite[p. 1921]{LLQRP}.

It is clear from the definition that $H^2(\beta)$ contains all analytic polynomials as a dense subset. The canonical orthonormal basis of $H^2(\beta)$ is given by $\{e_n^\beta\}_{n\in\mathbb{N}},$ where
\begin{equation}
\label{eq:enONBbeta}
e_n^\beta=\frac{z^n}{\sqrt{\beta_n}}, \qquad z\in\mathbb{D},\quad n\in\mathbb{N}.
\end{equation}

\subsection{Compact operators}
In this paper, we are interested in compactness and Schatten classes of composition operators. We follow \cite[Chapter 1]{Zhu} for a brief overview on this topic. Consider a compact operator $A$ on a Hilbert space $H$, i.e., an operator that sends bounded sets to relatively compact sets. It is known \cite[pp. 16-17]{Zhu} that $A$ can be written as
\[
Ax=\sum_{n=1}^\infty \lambda_n \langle x,e_n\rangle_{\!H} \sigma_n, \qquad x\in H,
\]
where $\{e_n\}_{n\geq 1},\{\sigma_n\}_{n\geq 1}$ are orthonormal sets and $(\lambda_n)_n$ is a non-increasing sequence of non-negative numbers that tend to $0$. The number $\lambda_n$ is called the $n$-th singular value of $A$, and it admits several other characterizations. For example, 
\[
\lambda_{n+1}=\inf\{\|A-F\|\colon F\in\mathcal{F}_{n}\},
\]
where $\norm{\cdot}$ is the standard operator norm and $\mathcal{F}_{n}$ is the set of  bounded linear operators whose range has dimension less than or equal to $n$; see \cite[Theorem 1.34]{Zhu}. Notice in particular that $\lambda_1=\|A\|$.

For $0< p<\infty,$ we say that $A$ belongs to the $p$-Schatten class $\mathcal{S}^p(H)$ if and only if 
\[
\|A\|_{\mathcal{S}^p}^p:=\sum_{n=1}^\infty
|\lambda_{n}|^p<\infty.
\] 
The space $(\mathcal{S}^p(H), \|\cdot\|_{\mathcal{S}^p})$ is a Banach space for $p\geq 1$; see \cite[p. 18]{Zhu}. Clearly, we have the inclusion $\mathcal{S}^p(H)\subseteq \mathcal{S}^q(H)$ for $p\leq q$. Within this family, $\mathcal{S}^1$ and $\mathcal{S}^2$ are also called, respectively, the trace class and the Hilbert--Schmidt class. 
The Hilbert--Schmidt class has the nice following property (see \cite[Theorem 1.22]{Zhu}): given any orthonormal basis $\{x_n\}_{n+1}$, 

\[
\sum_{n=1}^{\infty}|\lambda_n|^2=\sum_{n=1}^{\infty} \|Ax_n\|_H^2=\sum_{n,m=1}^{\infty}|\langle Ax_n,x_m\rangle|^2.
\]

In particular, in our context, given a composition operator $C_\varphi$ on $H^2(\beta),$ choosing the orthonormal basis as in \eqref{eq:enONBbeta}, we have that
\begin{align*}
   | \langle C_\varphi(e_n^\beta),e_m^\beta\rangle_\beta|^2&=\frac{\beta_m}{\beta_n}|\widehat{\varphi^n}(m)|^2, \qquad n,m\in\mathbb{N},
\end{align*}
where $\varphi^n$ denotes the usual power with respect to multiplication: $\varphi^n = \varphi \cdot \varphi \cdot \ldots \cdot \varphi.$ Then, we obtain the following direct characterization:
\begin{equation}
\label{eq:Hs-norm-comp-op}
\varphi \in \operatorname{HS}(\beta) \iff \|\varphi\|_{\operatorname{HS}(\beta)}^2=\sum_{n = 0}^{\infty}\dfrac{\|\varphi^n\|_\beta^2}{\beta_n}=\sum_{n,m = 0}^{\infty}\dfrac{\beta_m}{\beta_n}|\widehat{\varphi^n}(m)|^2 < +\infty.
\end{equation}
For the classical Hardy space, since $\beta=\mathbf{1}$, this leads to the identity \cite[p. 26]{Shapiro-Book}
\begin{equation}\label{eq:HSHardyintegral}
    \|\varphi\|_{\operatorname{HS}(\mathbf{1})}^2=\sum_{n = 0}^{\infty}\int_{\mathbb{T}}|\varphi|^{2n}\operatorname{d}\!m=\int_{\mathbb{T}}\frac{1}{1-|\varphi|^2}\operatorname{d}\!m,
\end{equation}
where $m$ is the (normalized) Lebesgue measure (i.e., $m(\T) = 1$).

For $p$-Schatten classes with $p\neq 2,$ the situation is more complicated: $A$ is in $\mathcal{S}^p(H)$ if and only if 
\[
\sum_{n,m}|\langle Ax_n,x_m\rangle|^p<\infty
\]
for \emph{every} orthonormal set $\{x_n\}_n$; see \cite[Theorem 1.27]{Zhu}. In other words, it no longer suffices to check the summability on a single basis. However, we will use the following argument for the trace class operators. This result should be well-known to experts, although we include its short proof.

\begin{lemma}
\label{lemma:S1}
Let $A\colon H\to H$ be a compact operator on a Hilbert space. If there exists an orthonormal basis $\{x_n\}_n$ of $H$ such that
\[
\sum_{n \in \N}\|Ax_n\|_H < +\infty,
\]
then $A$ is a trace class operator.
\begin{proof}
 For every $n \in \N$, consider the rank-one operator
\[
R_n h = \langle h,x_n\rangle_H Ax_n,  \qquad h \in H.
\]
Since a rank-one operator has only one non-zero singular value, equal to its operator norm, we have
\[
\norm{R_n}_{\mathcal{S}_1}
= \norm{R_n}
= \|Ax_n\|_H,
\qquad n \in \N.
\]
Then, the series $\sum_{n = 0}^{\infty}R_n$ converges both in the operator norm and in $\mathcal{S}^1(H)$.
We conclude with the simple observation that $A = \sum_{n = 0}^{\infty}R_n$. Indeed, if $h=\sum_{n=0}^\infty\langle h,x_n\rangle_{\!H}x_n$, then
\[
Ah=\sum_{n=0}^\infty\langle h,x_n\rangle_{\!H}Ax_n=\sum_{n=0}^\infty R_n h.
\]
In particular, using the triangle inequality,
\[
\norm{A}_{\mathcal{S}^1} \leq \sum_{n \in \N} \norm{R_n}_{\mathcal{S}^1} = \sum_{n \in \N}\|Ax_n\|_H < +\infty.
\]
\end{proof}
\end{lemma}

\section{Proof of Theorem \ref{T:main1}}
\label{sec:proof1}

In this section we prove Theorem \ref{T:main1} using the combination of two results. In the first of them we show that our main questions (\textbf{Q1}) and (\textbf{Q2}) are easy to solve if the weight $\beta$ is unbounded. Similar remarks already appeared for bounded composition operators in \cite[Proposition 3.1]{LLQRP}, \cite[Proposition 3.1]{Zorboska}, and \cite[Proposition~2]{Zorboska-PAMS}.

\begin{proposition}
\label{P:betabounded}
Let $\beta$ be a weight. If $\beta$ is not bounded, then there exists a holomorphic map $\varphi \colon \D \to \D$ such that $\norm{\varphi}_{\infty} < 1$ and $C_\varphi$ is unbounded on $H^2(\beta)$. In particular, $\varphi \in \operatorname{S}_1(\mathbf{1})$ but $\varphi \not\in \operatorname{K}(\beta)$.
\begin{proof} 
Since $\beta$ is unbounded, we may find an increasing subsequence $(\beta_{n_k})_{k \in \N}$ such that $\beta_{n_k} \geq (k+1)^3$, $k \in \N$. Consider the sequence of coefficients $(a_n)_{n \in \N}$ given by
\[
a_n = 
\begin{cases}
\dfrac{c}{(k+1)^2}, & \text{if } n = n_k, \\
0, & \mathrm{otherwise},
\end{cases}
\]
where $c > 0$ is such that
\[
\sum_{n = 0}^{\infty}a_n = \sum_{k=0}^{\infty}\dfrac{c}{(k+1)^2} < \frac12.
\]
Then, the map given by
$\varphi(z) = \sum_{n = 0}^{\infty}a_nz^n$ satisfies $\|\varphi\|_{\infty}\leq\frac12<1$. Then,
$$\sum_{n \in \N}\norm{\varphi^n}_2 = \sum_{n \in \N} \norm{C_{\varphi}e_n}_2  \leq \sum_{n \in \N}\dfrac{1}{2^n} < +\infty.$$
By Lemma \ref{lemma:S1}, $C_{\varphi}$ is of trace class in $H^2$ (i.e., $\varphi \in \operatorname{S}_1(\mathbf{1})$) (see also \cite[Proposition~1]{Zorboska-PAMS}). However, $\varphi = C_\varphi z \notin H^2(\beta)$ since
\[
\norm{\varphi}_{\beta}^2 = \sum_{n = 0}^{\infty}\beta_na_n^2 = \sum_{k = 0}^{\infty}c^2\dfrac{\beta_{n_k}}{(k+1)^{4}} \geq \sum_{k = 0}^{\infty}c^2\dfrac{(k+1)^3}{(k+1)^4}= +\infty.
\]
In particular, this means that $C_{\varphi}$ is not bounded on $\mathrm{H^2}(\beta)$ (and so $\varphi \not\in \operatorname{K}(\beta)$).
\end{proof}
\end{proposition}

In our second result we analyze the opposite situation. Namely, we consider weights converging to zero. In this scenario, with the aid of Lemma \ref{lemma:Jones}, we explicitly construct a Blaschke product which gives rise to a Hilbert--Schmidt composition operator on $H^2(\beta)$. Inner functions can never be symbols of compact composition operators on $H^2$. Indeed, if $u$ is inner, then $\|C_u(z^n)\|_2=\|u^n\|_2=1$ for every $n\in\mathbb{N}$, but $z^n$ converges weakly to $0$ in $H^2$ (see \cite[Theorem 1.14]{Zhu}).

\begin{proposition}
\label{prop::limzero}
Let $\beta$ be a weight such that $\lim_{n \to \infty}\beta_n = 0$. Then, there exists a Blaschke product $B \in \operatorname{S}_1(\beta)$. In particular, $B \not\in \operatorname{K}(\mathbf{1})$.
\begin{proof}
By hypothesis, $\beta$ is bounded. Consider $M := \sup_{n \in \N} \beta_n.$
Since we are also assuming \eqref{eq:cond-D}, we deduce that 
\[
\lim_{n \to \infty} \beta_n^{1/n} = 1.
\]
Now, consider $\varepsilon_n := \beta_n/(n+1)^4$, $n \in \N$. Notice that we also have
\begin{equation}
\label{eq::eps-1}
\lim_{n \to \infty} \varepsilon_n^{1/n} = 1.
\end{equation}
Since $\limsup_{n \to \infty} \beta_n = 0$, we can find an increasing sequence $(L_n)_{n \in \N}$ in $\mathbb{N}$ such that
\begin{equation}
\label{eq::Ln}
\sup_{m \geq L_n} \beta_m \leq \dfrac{\varepsilon_n}{2}, \qquad n\in\mathbb{N}.
\end{equation}
Set
\begin{equation}
\label{eq::rn}
r_n:=  2^{-\frac{1}{2L_n}}, \qquad n\in\mathbb{N}.
\end{equation}
This sequence satisfies $r_n<1$ for every $n\in\mathbb{N}$ and $\lim_{n \to \infty} r_n=1$.
For each $n \in \N$, let us also consider
\begin{equation}
\label{eq::qn}
q_n := \min\left\lbrace 1-\dfrac{1}{n+1}, \left(\dfrac{\varepsilon_n}{4M}\right)^{\tfrac{1}{2n}}\right\rbrace.
\end{equation}
The sequence $(q_n)_{n \in \N}$ may not be non-decreasing. For this reason, for each $n \in \N$ we consider also
\[
s_n = \inf_{m \geq n}q_m.
\]
By \eqref{eq::eps-1}, we have $\lim_{n \to \infty}s_n = 1$. Using this, we may consider a continuous non-decreasing function $w \colon [r_0,1] \to (0,1]$ with $w(1) = 1$ and $w(r_n) = s_n$. We now use Lemma \ref{lemma:Jones} to find a Blaschke product $B$ with
\begin{equation}
\label{eq::MrnB}
M(r_n,B) < w(r_n) = s_n \leq q_n, \qquad n \in \N.
\end{equation}

To conclude, we need to justify that $B \in \operatorname{S}_1(\beta)$. For every $n \in \N$, $B^n$ is an inner function. In particular, we have
\begin{equation}
\label{eq::B-sum-inner}
\sum_{m \in \N}|\widehat{B^n}(m)|^2 = \|B^n\|_2^2= 1.
\end{equation}
We want to estimate $\norm{B^n}_{\beta}$. Using \eqref{eq::rn},
\[
\sum_{m < L_n}|\widehat{B^n}(m)|^2  \leq r_n^{-2L_n}\sum_{m < L_n}|\widehat{B^n}(m)|^2r_n^{2m} \leq 2\sum_{m \in \N}|\widehat{B^n}(m)|^2r_n^{2m}.
\]
Now, notice that
\[
\sum_{m \in \N}|\widehat{B^n}(m)|^2r_n^{2m}=\|B^n(r_n\cdot)\|^2_2=\frac{1}{2\pi}\int_0^{2\pi}|B^n(r_ne^{i\theta})|^2\operatorname{d}\!\theta\leq M(r_n,B)^{2n}.
\]
By construction, using \eqref{eq::qn} and \eqref{eq::MrnB}, 
\begin{align}
\label{eq::sum-Ln}
\sum_{m < L_n}|\widehat{B^n}(m)|^2 \leq 2M(r_n,B)^{2n}\leq 2q_n^{2n}\leq \dfrac{\varepsilon_n}{2M}.
\end{align}
Now, combining \eqref{eq::Ln}, \eqref{eq::B-sum-inner}, and \eqref{eq::sum-Ln} we see that for every $n$ in $\N$
\[
\norm{B^n}_\beta^2
= 
\sum_{m < L_n}\beta_m|\widehat{B^n}(m)|^2 + \sum_{m \geq L_n}\beta_m|\widehat{B^n}(m)|^2
\leq 
\frac{M\varepsilon_n}{2M} + \frac{\varepsilon_n}{2}  \sum_{m=0}^\infty|\widehat{B^n}(m)|^2 = \varepsilon_n.
\]
Therefore,
\[
\sum_{n \in \N}\dfrac{\norm{B^n}_\beta}{\sqrt{\beta_n}} \leq \sum_{n \in \N} \dfrac{\sqrt{\varepsilon_n}}{\sqrt{\beta_n}} = \sum_{n \in \N}\dfrac{1}{(n+1)^2} < +\infty.
\]
By Lemma \ref{lemma:S1}, $B \in \operatorname{S}_1(\beta)$.

As remarked above, $B \not\in \operatorname{K}(\mathbf{1})$ since it is an inner function. 
\end{proof}
\end{proposition}

We now see that Theorem \ref{T:main1} simply follows from the combination of Propositions \ref{P:betabounded} and \ref{prop::limzero}. Also, Corollary \ref{cor:h2} follows from the combination of Theorems \ref{T:rodriguezPiazza} and \ref{T:main1}.

\section{Lacunary weights}
\label{sec:proof2}
As we motivated in the Introduction, we consider lacunary weights to better understand what can happen for general bounded weights $\beta$ for which $0=\liminf_{n \to \infty}\beta_n<\limsup_{n\to\infty}\beta_n<+\infty$. 

We consider a sequence of natural numbers $\Lambda=\{N_k\}_{k \in \N}$ with the property that there exists $q>1$ such that $N_{k+1}\geq q N_k$ for all $k \in \N$. The following powerful result due to Paley \cite{Paley} is crucial for our purposes.

\begin{theorem}
\label{thm:Paley}
There exists $C_{\Lambda} > 0$ such that, for all $f \in H^1$, we have
\[
\sum_{n \in \Lambda}|\widehat{f}(n)|^2 \leq C^2_{\Lambda}\norm{f}_{H^1}^2.
\]
\end{theorem}
This result also appears in \cite[Section 6.4]{Duren-Hp}. However, the statement above is not exactly the one given in \cite{Duren-Hp,Paley}. In these references, it is simply stated that if $\norm{f}_{H^1} < +\infty$, then $\sum_{n \in \Lambda}|\widehat{f}(n)|^2 < +\infty$, but the existence of $C_{\Lambda} > 0$ is shown in the proof (not the statement) of \cite{Paley}. In this form, the result appears in \cite[Eq. (1.9)]{Sola} for $N_k = 2^k$, for example.

Motivated by this result of Paley, we introduce the (bounded) projection operators $P \colon H^1 \to H^2$ and $Q \colon H^2\to H^2$ given by
\begin{align}
    \label{eq:defP} Pf(z) &= \sum_{n \in \Lambda}\widehat{f}(n)z^{n}, \qquad z \in \D, \quad f\in H^1, \\
     \label{eq:defQ} Qf &= f - Pf, \qquad f \in H^2.
\end{align}
On the Hilbert space $H^2$ the operators $P,Q$ are orthogonal projections. We will repeatedly use that
$$\varphi^{n} = P(\varphi^n)+Q(\varphi^n), \qquad n \in \N,$$
for the symbols $\varphi \colon \D \to \D$. We remark that in general $Q$ is not well-defined on $H^1$.

\subsection{Compactness}
We begin with a general lemma that allows one to construct compact composition operators on $H^2$ based on the inner-outer factorization.

\begin{lemma}\label{L:compactsingularinner}
    Let $u$ be a non-negative, continuous function on $[-\pi,\pi]$, with the property that $m_\varepsilon=\min_{\varepsilon\leq |t|\leq \pi}u(t)>0$ for every $\varepsilon>0$. Then, the symbol
    $\varphi \colon \D \to \D$ given by
\begin{equation}
\label{eq:def-varphi}
\varphi(z) = \exp\left(-F(z)-\dfrac{1+z}{1-z}\right), \qquad z \in \D,
\end{equation}
where
\begin{equation}
\label{eq:def-F}
F(z) = \dfrac{1}{2\pi}\int_{-\pi}^{\pi}\dfrac{e^{it}+z}{e^{it}-z}\,u(t)\operatorname{d}\!t, \qquad z \in \D,
\end{equation}
defines a compact operator $C_\varphi\colon H^2\to H^2$.
\end{lemma}

\begin{proof}
By assumption $u$ is non-negative and continuous, and  $\varphi$ is the product of the outer function $\exp{(-F)}$ with a singular inner function. In particular, 
\[
\|\varphi\|_{\infty}\leq \|\exp{(-F)}\|_{\infty} = \|e^{-u}\|_{L^\infty}\leq 1,
\]
and the symbol $\varphi\colon\mathbb{D}\to\mathbb{D}$ satisfies  $|\varphi(e^{it})| = e^{-u(t)}$ for a.e.  $t \in [-\pi,\pi]$.

By \cite[Corollary 2.4]{Akeroyd} (see also \cite{CimaMatheson}), to prove that $C_\varphi\colon H^2\to H^2$ is compact it suffices to prove that
\[
\lim_{|a| \to 1^-}\norm{k_a \circ \varphi}_{2} = 0,
\]
where, for every fixed $a\in\mathbb{D}$,
\[
k_a(z) = \frac{\sqrt{1-|a|^2}}{1-\overline{a}z}, \qquad z \in \D.
\]

Fix $0<\varepsilon<1$. We split the $H^2$-norm as follows:
\[
\norm{k_a\circ\varphi}_2^2
= 
\frac{1}{2\pi}\int_{|t|<\varepsilon}
\frac{1-|a|^2}{\abs{1-\overline{a}\varphi(e^{it})}^2}\operatorname{d}\!t
+ 
\frac{1}{2\pi}\int_{\varepsilon\leq|t|\leq\pi}
\frac{1-|a|^2}{\abs{1-\overline{a}\varphi(e^{it})}^2}\operatorname{d}\!t.
\]

We estimate the second integral. By assumption, we have that
\[
|\varphi(e^{it})| = e^{-u(t)}\leq e^{-m_\varepsilon}\leq 1,
\qquad \varepsilon\leq|t|\leq\pi,
\]
for almost every $t$. Therefore,
\[
|1-\overline{a}\varphi(e^{it})|
\geq 
1-|a||\varphi(e^{it})|
\geq 
1-e^{-m_\varepsilon},
\]
from which it follows that
\[
\frac{1}{2\pi}\int_{\varepsilon\leq|t|\leq\pi}
\frac{1-|a|^2}{|1-\overline{a}\varphi(e^{it})|^2}\operatorname{d}\!t
\leq \frac{1-|a|^2}{(1-e^{-m_\varepsilon})^2}.
\]

In order to estimate the first integral, consider the positive harmonic function
\[
h_a(z)=\operatorname{Re}\left(
\frac{1+\overline{a}\varphi(z)}{1-\overline{a}\varphi(z)}
\right)
=\frac{1-|a|^2|\varphi(z)|^2}
{|1-\overline{a}\varphi(z)|^2},
\qquad z\in\D.
\]
Since $|\varphi(z)|\leq 1$, we have
\[
\abs{k_a(\varphi(z))}^2
=\frac{1-|a|^2}{\abs{1-\overline{a}\varphi(z)}^2}
\leq h_a(z),\qquad z\in\D.
\]
The same inequality holds for almost every boundary point. Moreover, for each fixed $a\in\D$, the function $h_a$ is bounded, since
\[
0<h_a(z)\leq
\frac{1+|a||\varphi(z)|}{1-|a||\varphi(z)|}
\leq\frac{1+|a|}{1-|a|}.
\]
Consequently, the Poisson formula gives
\begin{equation}
\label{eq:formula-int-ha}
h_a(r)=\frac{1}{2\pi}\int_{-\pi}^{\pi}
\frac{1-r^2}{\abs{e^{it}-r}^2}\,h_a(e^{it})\operatorname{d}\!t,
\qquad 0\leq r<1.
\end{equation}
Choose $r=1-\varepsilon$. For $|t|<\varepsilon$, using $2(1-\cos t)\leq t^2$, we obtain
\[
|e^{it}-r|^2
=(1-r)^2+2r(1-\cos t)
\leq\varepsilon^2+t^2
\leq 2\varepsilon^2.
\]
On the other hand,
\[
1-r^2=\varepsilon(2-\varepsilon)\geq\varepsilon.
\]
Thus,
\[
\frac{1-r^2}{\abs{e^{it}-r}^2}
\geq
\frac{1}{2\varepsilon},
\qquad |t|<\varepsilon.
\]
Using this information in \eqref{eq:formula-int-ha}, we have
\[
h_a(r) \geq 
\dfrac{1}{2\varepsilon}\dfrac{1}{2\pi}\int_{|t| < \varepsilon} h_a(e^{it})\operatorname{d}\!t.
\]
This further implies that
\[
\frac{1}{2\pi}\int_{|t|<\varepsilon}
\frac{1-|a|^2}{\abs{1-\overline{a}\varphi(e^{it})}^2}\operatorname{d}\!t
\leq
\frac{1}{2\pi}\int_{|t|<\varepsilon} h_a(e^{it})\operatorname{d}\!t
\leq 
2\varepsilon h_a(1-\varepsilon).
\]

It remains to bound $h_a(1-\varepsilon)$ independently of $a$ and $\varepsilon$. Since $u$ is non-negative, \eqref{eq:def-F} gives
\[
\operatorname{Re}F(r)
=
\frac{1}{2\pi}\int_{-\pi}^{\pi}
\frac{1-r^2}{\abs{e^{it}-r}^2}\,u(t)\operatorname{d}\!t
\geq 0.
\]
By \eqref{eq:def-varphi}, this implies that
\[
|\varphi(r)|
=
\exp\left(-\operatorname{Re}F(r)-\frac{1+r}{1-r}\right)
\leq
\exp\left(-\frac{1+r}{1-r}\right)
\leq 
e^{-1},\qquad 0\leq r<1.
\]
Therefore,
\[
h_a(1-\varepsilon)
\leq
\frac{1+|a||\varphi(1-\varepsilon)|}
{1-|a||\varphi(1-\varepsilon)|}
\leq
\frac{1+e^{-1}}{1-e^{-1}}.
\]

Combining the estimates above, we conclude that
\[
\norm{k_a\circ\varphi}_2^2
\leq 
2\varepsilon\frac{1+e^{-1}}{1-e^{-1}}
+\frac{1-|a|^2}{(1-e^{-m_\varepsilon})^2}.
\]
For each fixed $\varepsilon$, the second term tends to zero as $|a|\to 1^-$. Hence,
\[
\limsup_{|a|\to 1^-}
\norm{k_a\circ\varphi}_2^2
\leq 
2\varepsilon\frac{1+e^{-1}}{1-e^{-1}}.
\]
Letting $\varepsilon\to 0$ proves the desired limit, and the compactness of $C_\varphi\colon H^2\to H^2$.
\end{proof}

As a consequence of this result, we are able to characterize the equality $\operatorname{K}(\beta)=\operatorname{K}(\mathbf{1})$. In particular, the characterization is the same as in the case of $\beta$ essentially decreasing. 

\begin{theorem}\label{T:lacunaryKsec4}
Let $\beta$ be a weight as in \eqref{E:beta_lacunary}. 
Then, $\mathrm{K}(\beta) = \mathrm{K}(\mathbf{1})$ if and only if $\beta \simeq \mathbf{1}$ (i.e., $\inf_{k \in \N}\beta_{N_k} > 0$).
\end{theorem}
\begin{proof}
If $\beta \simeq \mathbf{1}$, it is clear that $\mathrm{K}(\beta) = \mathrm{K}(\mathbf{1})$. Let us assume that $\inf_{k \in \N}\beta_{N_k} = 0$. Using Lemma \ref{L:compactsingularinner}, we will construct $\varphi \in \mathrm{K}(\mathbf{1})$ such that $C_{\varphi}$ is unbounded on $H^2(\beta)$. Since $\inf_{k \in \N}\beta_{N_k} = 0$, choose a subsequence $(M_k)_{k \in \N}$ of $(N_k)_{k \in \N}$ such that
\begin{equation}
\label{eq:def-Mk}
\frac{1}{\beta_{M_k}} \geq 4^k, \qquad k \in \N.
\end{equation}
Now, consider the partition of $[0,\pi]$ given by the points $\{t_k\}_{k\in\mathbb{N}}$, where $t_k=2\arcsin(2^{-k}), \, k \in \N$. Notice that $t_0=\pi$, $t_{k+1}<t_k$, and $t_k\to 0$ as $k\to\infty$. Associated to this choice, we can find an even continuous function $u \colon [-\pi,\pi] \to [0,+\infty)$ such that $u(t) > 0$ whenever $t\neq 0$, $u$ is increasing in $[0,\pi]$, and
\begin{equation}
\label{eq:prop-u}
u(t) \leq \dfrac{1}{M_k}, \qquad |t| \leq t_k \qquad k \in \N.
\end{equation}

This function $u$ satisfies the assumptions of Lemma \ref{L:compactsingularinner}. In particular, the symbol 
 \[
 \varphi(z)=\exp\left(-\dfrac{1}{2\pi}\int_{-\pi}^{\pi}\dfrac{e^{it}+z}{e^{it}-z}\,u(t)\operatorname{d}\!t-\dfrac{1+z}{1-z}\right), \qquad z \in \D,
 \]
induces a compact operator $C_\varphi\colon H^2\to H^2$. We show that $C_\varphi$ is unbounded on $H^2(\beta)$. Heuristically, the crucial property that $\varphi$ satisfies is that $\|Q(\varphi^n)\|_2\simeq \|\varphi^n\|_2$, that is, every power of $\varphi$ has enough mass outside of the lacunary set $\Lambda$. 

We work out the details with more precision. For $n \in \N$, 
\[
\norm{\varphi^n}_{H^1} 
= 
\dfrac{1}{2\pi}\int_{-\pi}^{\pi}e^{-nu(t)}\operatorname{d}\!t, \qquad 
\norm{\varphi^n}^2_{2} 
= 
\dfrac{1}{2\pi}\int_{-\pi}^{\pi} e^{-2nu(t)} \operatorname{d}\!t.
\]
Fix $0<\varepsilon<\pi$. Splitting the first integral at $|t|=\varepsilon$ and using Cauchy--Schwarz, since $u$ is even and increasing, gives
\[
\norm{\varphi^n}_{H^1} 
\leq \sqrt{\varepsilon/\pi}\,\norm{\varphi^n}_2 + e^{-nu(\varepsilon)}.
\]

Notice that, by construction, $u(0)=\lim_{k\to\infty}u(t_k)=0.$ By continuity at zero, choose $0<\delta<\varepsilon$ with $u(t)\le u(\varepsilon)/2$ for every $|t|<\delta$. Then,
\begin{align*}
    \norm{\varphi^n}^2_{2}
\ge
\dfrac{1}{2\pi}\int_{|t|<\delta} e^{-2nu(t)} \operatorname{d}\!t
\ge 
\frac\delta\pi e^{-nu(\varepsilon)},
\end{align*}
from which we obtain
\[
\frac{\norm{\varphi^n}_{H^1}}{\norm{\varphi^n}_2}
\le
\sqrt{\varepsilon/\pi}+\sqrt{\pi/\delta}\,e^{-nu(\varepsilon)/2}.
\]
Since $u(\varepsilon)>0$, taking a limsup as $n \to \infty$ and then letting $\varepsilon \to 0$ proves that
\[
\lim_{n \to \infty}\dfrac{\norm{\varphi^n}_{H^1}}{\norm{\varphi^n}_2} = 0.
\]

Paley's inequality now gives $\norm{P(\varphi^n)}_2\le C_\Lambda\norm{\varphi^n}_{H^1}=o(\norm{\varphi^n}_2)$. Using the orthogonality of $P$ and $Q$, this proves that for $n$ bigger than some $N_0$,
\begin{align}\label{eq:normQnormphi}
    \norm{Q(\varphi^n)}_2^2 = \norm{\varphi^n}^2_2-\norm{P(\varphi^n)}_2^2\geq \frac{1}{4}\norm{\varphi^n}^2_2.
\end{align}

We now prove that $C_\varphi$ is unbounded on $H^2(\beta)$. Recall that the monomials $e_{n}^{\beta}=z^{n}/\sqrt{\beta_{n}}$ have norm one in $H^2(\beta)$. We show that $\limsup_{n\to\infty}\|C_\varphi e_n^\beta\|_\beta=+\infty.$ By \eqref{eq:normQnormphi} and the definition of $\beta$ and $Q$,
\begin{equation*}
\label{eq:norm-estimate}
\norm{\varphi^n}_\beta 
\geq \norm{Q(\varphi^n)}_\beta 
\geq
c\norm{Q(\varphi^n)}_{2} 
\geq 
\frac{c}{2}
\norm{\varphi^n}_{2},\qquad n\geq N_0,
\end{equation*} 
where $c = \inf_{n \not\in \Lambda}\beta_n > 0$.

By \eqref{eq:def-Mk} and \eqref{eq:prop-u}, considering the subsequence $(M_k)_k$,
\[
\norm{\varphi^{M_k}}_2^2\ge\int_{-t_k}^{t_k}e^{-2M_ku(t)}\,\frac{\operatorname{d}\!t}{2\pi}\ge e^{-2}\frac{t_k}{\pi}\ge\frac{2e^{-2}}\pi\,2^{-k},
\]
where we have also used that $\arcsin(x) \geq x$ for all $0 \leq x \leq 1$. Then, by \eqref{eq:def-Mk} we conclude that, for big enough $k$,
\[
\norm{C_\varphi e_{M_k}^{\beta}}_\beta^2
=
\frac{\norm{\varphi^{M_k}}_\beta^2}{\beta_{M_k}}
\geq
\frac{c\norm{\varphi^{M_k}}_2^2}{2\beta_{M_k}}
\ge
\frac2c \frac{2e^{-2}}{\pi}2^k\xrightarrow[k\to\infty]{}+\infty.
\]
\end{proof}

\subsection{Hilbert--Schmidt}
In this section, we characterize the equality 
$\operatorname{HS}(\beta)=\operatorname{HS}(\mathbf{1})$, under the lacunarity condition. In this case, the assumption $\beta \simeq \mathbf{1}$ breaks down: we can have equality between the classes of Hilbert--Schmidt composition operators when $\liminf_{n\to\infty}\beta_n=0$.

We will use the following property of the Paley projection $P$. For what follows, we need to assume that the constant $q$ for our lacunary set $\Lambda$ satisfies $q>2$.

\begin{lemma}
\label{lemma:projection-squared}
Assume that $q > 2$. Then, $P[(Pf)^2] = 0$ for all $f \in H^1$.
\end{lemma}
\begin{proof}
Notice that we can compute $P[(Pf)^2]$ since $(Pf)^2\in H^1$. Now, let $n,m \in \N$ with $m \geq n$. Then, since $q>2,$ we have that $N_n+N_m \leq 2N_m < N_k$ for all $k > m$. This means that $N_n+N_m \neq N_k$ for all $n,m,k \in \N$. Therefore, $P[z \mapsto z^{N_n}z^{N_m}] = 0$ for all $n,m \in \N$, from which the result follows at once.
\end{proof}

We first show that a set inequality holds without extra assumptions.

\begin{proposition}
Consider $\beta$ as in \eqref{E:beta_lacunary}, where $q > 2$. Then, $\mathrm{HS}(\beta) \subseteq \mathrm{HS}(\mathbf{1})$. More precisely, there exists a constant $C>0$ such that
\[
\|\varphi\|_{\operatorname{HS}(\mathbf{1})}\leq C \|\varphi\|_{\operatorname{HS}(\beta)}, \qquad \varphi\in \operatorname{HS}(\beta).
\]
\end{proposition}

\begin{proof}
Let $\varphi \in \mathrm{HS}(\beta)$. We need to show that
\[
\norm{\varphi}^2_{\mathrm{HS}(\mathbf{1})} = \sum_{n \in \N}\norm{\varphi^n}_2^2 < + \infty.
\]
Using the projection operators introduced before, since $\varphi^n\in H^2$, 
\[
\norm{\varphi^n}_2^2 = \norm{P(\varphi^n)}_2^2 + \norm{Q(\varphi^n)}_2^2, \qquad n\in\mathbb{N}.
\]

By the definition of $\beta$ and $Q$, we have that
\begin{equation}
\label{eq:sum-Q}
\sum_{n \in \N}\norm{Q(\varphi^n)}^2_2 \leq \sum_{n \in \N}\norm{\varphi^n}^2_\beta \leq \sum_{n \in \N}\dfrac{\norm{\varphi^n}^2_\beta}{\beta_n} = \norm{\varphi}^2_{\mathrm{HS}(\beta)} < +\infty.
\end{equation}

It remains to show that $\sum_{n \in \N}\norm{P(\varphi^n)}^2_2 < +\infty$. In order to do so, using Lemma~\ref{lemma:projection-squared}, we see that
\begin{align*}
P(\varphi^{2n}) & = P[(\varphi^n)^2] = P\big[(P(\varphi^n)+Q(\varphi^n))^2\big] \\
& = P\big[(P(\varphi^n))^2+(Q(\varphi^n))^2+2P(\varphi^n)Q(\varphi^n)\big] \\
& = P\big[(Q(\varphi^n))^2+2P(\varphi^n)Q(\varphi^n)\big].
\end{align*}
Then, by Theorem \ref{thm:Paley}, we have that
\[
\norm{P(\varphi^{2n})}_2
 \leq C_\Lambda\norm{(Q(\varphi^n))^2}_{H^1}+2C_\Lambda\norm{P(\varphi^n)Q(\varphi^n)}_{H^1}. 
\]
Next, since $\norm{P(\varphi^n)}_2 \leq \norm{\varphi^n}_2 \leq 1$, and similarly for $Q$, we have that
\begin{align*}
\norm{P(\varphi^{2n})}_2 &\leq C_\Lambda\norm{Q(\varphi^n)}^2_2+2C_\Lambda\norm{P(\varphi^n)}_2\norm{Q(\varphi^n)}_2 \\
& \leq C_\Lambda\norm{Q(\varphi^n)}^2_2+2C_\Lambda\norm{Q(\varphi^n)}_2 \\
& \leq 3C_\Lambda\norm{Q(\varphi^n)}_2.
\end{align*}
 Combining this inequality with \eqref{eq:sum-Q}, we deduce that
\[
\sum_{n \in \N}\norm{\varphi^{2n}}_2^2 
\leq
\sum_{n\in\mathbb{N}}(9 C_\Lambda^2 \|Q(\varphi^{n})\|^2+\|Q(\varphi^{2n})\|^2)
\leq 
(9C_\Lambda^2+1)\|\varphi\|_{\operatorname{HS}(\beta)}^2.
\]
Finally, since
\[
\norm{P(\varphi^{2n+1})}_2
\leq 
C_\Lambda\norm{\varphi^{2n+1}}_2 
\leq
C_\Lambda \norm{\varphi^{2n}}_2,
\]
we conclude that
\[
\|\varphi\|_{\operatorname{HS}(\mathbf{1})}^2
=
\sum_{n \in \N}\norm{\varphi^n}_2^2
\leq
C\|\varphi\|_{\operatorname{HS}(\beta)}^2
< +\infty,
\]
where the constant $C = 10C_{\Lambda}^2+2$ does not depend on $\varphi$. 
\end{proof}

We complete the result with a characterization for the reverse set inequality.

\begin{proposition}\label{P:secondinclusionHSlacunary}
Consider $\beta$ as in \eqref{E:beta_lacunary}, where $q > 2$. Then, $\mathrm{HS}(\mathbf{1})\subseteq \mathrm{HS}(\beta)$ if and only if $\inf_{k \in \N}\beta_{N_k}N_k>0$. Moreover, in this case, there exists a constant $C>0$ such that
\[
\|\varphi\|_{\operatorname{HS}(\beta)}\leq C\|\varphi\|_{\operatorname{HS}(\mathbf{1})}, \qquad\varphi\in\operatorname{HS}(\mathbf{1}).
\]
 
\end{proposition}

\begin{proof}
First, we assume that $\inf_{k \in \N}\beta_{N_k}N_k>0$. We consider $\varphi \in \mathrm{HS}(\mathbf{1})$, that is,
\[
\|\varphi\|_{\operatorname{HS}(1)}^2=\sum_{n \in \N}\norm{\varphi^n}^2_2 < +\infty,
\]
and we prove that $\varphi\in\operatorname{HS}(\beta)$.

Since we are assuming that $c\leq \beta_n\leq 1$ uniformly for $n \not\in \Lambda$, we have that
\begin{align}
\label{eq:estimate-norm-HSbeta}
\norm{\varphi}^2_{\mathrm{HS}(\beta)} & = \sum_{n \in \N}\dfrac{\norm{\varphi^n}^2_\beta}{\beta_n} \leq \sum_{n \in \N} \dfrac{\norm{\varphi^n}^2_2}{\beta_n} \\
& 
\leq \dfrac{1}{c}\sum_{n \in \N \setminus \Lambda}\norm{\varphi^n}^2_2 + \sum_{n \in \Lambda}\dfrac{1}{\beta_n}\norm{\varphi^n}^2_2 \notag \\
& 
\leq 
\dfrac{1}{c} \|\varphi\|_{\operatorname{HS}(\mathbf{1})}^2 +  \sum_{k\in\mathbb{N}}\frac{1}{\beta_{N_k}}\norm{\varphi^{N_k}}^2_2 \notag.
\end{align}

By assumption, 
\[
C_1 :=\inf_{k}\beta_{N_k} N_k >0.
\]
Then, 
\begin{align*}
      \sum_{k\in\mathbb{N}}\frac{1}{\beta_{N_k}}\norm{\varphi^{N_k}}^2_2 
      &\leq
      \frac{1}{C_1} \sum_{k\in\mathbb{N}}N_k\norm{\varphi^{N_k}}^2_2.
\end{align*}

Now, notice that $\norm{\varphi^{n}}^2_2$ is a decreasing function of $n$. Then,
\[
N_0\norm{\varphi^{N_0}}^2_2 \leq \sum_{n=0}^{N_0}\norm{\varphi^n}^2_2.
\]
Similarly, for each $k \in \N$ we have
\[
N_{k+1}\norm{\varphi^{N_{k+1}}}^2_2 \leq \dfrac{q}{q-1}(N_{k+1}-N_k)\norm{\varphi^{N_{k+1}}}^2_2 \leq \dfrac{q}{q-1}\sum_{n=N_k+1}^{N_{k+1}}\norm{\varphi^n}^2_2.
\]
All in all,
\[ \sum_{k\in\mathbb{N}}N_k\norm{\varphi^{N_k}}^2_2
\leq
\frac{q}{q-1} \sum_{n = 0}^{\infty}\norm{\varphi^{n}}^2_2 < +\infty.
\]

In particular, we have proved 
\[
\|\varphi\|_{\operatorname{HS}(\beta)}\leq C \|\varphi\|_{\operatorname{HS}(\mathbf{1})},
\]
where $C$ does not depend on $\varphi$.

We move on to the necessity part. Assume that $\mathrm{HS}(\mathbf{1})\subseteq \mathrm{HS}(\beta)$. We prove that $\inf_{k \in \N}\beta_{N_k}N_k > 0$ by contradiction: we suppose that $\inf_{k \in \N}\beta_{N_k}N_k=0$, and we construct $\varphi\in \operatorname{HS}(\mathbf{1})\setminus\operatorname{HS}(\beta)$.

Assuming that $\inf_{k \in \N}\beta_{N_k}N_k = 0$, we may choose an increasing subsequence $(M_k)_{k \in \N}$ of $(N_k)_{k \in \N}$ such that
\begin{equation}
\label{eq:choice-1}
\dfrac{1}{\beta_{M_k}{M_k}} \geq 4^k, \qquad k \in \N.
\end{equation}
By construction, $M_{k+1} \geq qM_k > 2M_k$. Therefore,
\[
\lim_{k \to \infty} \frac{2^kM_k}{2^{2M_k-1}}=0.
\]
In particular, for every $\varepsilon>0$ there exists $S_\varepsilon\in\mathbb{N}$ such that
\begin{equation}\label{eq:choice-2}
\frac{C_{\Lambda}^2}{2^{2M_k-1}}\leq \frac{\varepsilon}{16 \cdot 2^kM_k}, \qquad k\geq S_\varepsilon,
\end{equation}
where $C_{\Lambda}$ is as in Theorem \ref{thm:Paley}. The value $\varepsilon > 0$ 
will be fixed later. 

Now, choose pairwise disjoint measurable sets $E_k \subseteq \partial\D$ such that
\[
m(E_k) = \dfrac{\varepsilon}{2^kM_k}, \qquad k \in \N,
\]
where $m$ denotes normalized Lebesgue measure on $\partial\D$. Here, we assume that $\varepsilon > 0$ is small enough so that $m(E) \in (0,1)$, where $E = \cup_{k \in \N}E_k$.

Define a measurable function $\rho \colon \partial\D \to (0,1)$ by
\begin{equation}
\label{eq:rho}
\rho^2 =
\begin{cases}
1-\dfrac{1}{M_k}, & \text{on } E_k,\\
\dfrac{1}{4}, & \text{on } \partial\D \setminus E.
\end{cases}
\end{equation}
Notice that $\log(\rho)\in L^\infty(\partial\D)$ and $\rho < 1$ on $\partial\D$. Therefore, there exists an outer function $\varphi \colon \D \to \D$ such that $\abs{\varphi}=\rho$ almost everywhere on
$\partial\D$. This function $\varphi$ satisfies $\varphi\in \operatorname{HS}(\mathbf{1})$: indeed, using \eqref{eq:HSHardyintegral},
\[
\int_{\partial\D}\dfrac{1}{1-|\varphi|^2}\operatorname{d}\!m = \dfrac{4}{3}(1-m(E))+\sum_{k \in \N}m(E_k)M_k = \dfrac{4}{3}(1-m(E)) + \varepsilon \sum_{k \in \N}\dfrac{1}{2^k} < +\infty.
\]

We prove that $\varphi \not\in \mathrm{HS}(\mathbf{\beta})$, whence the contradiction. In particular, we prove that there exist choices of $\varepsilon$ for which
\[
\sum_{k\geq S_\varepsilon}\frac{\|\varphi^{M_k}\|^2_\beta}{\beta_{M_k}}=+\infty.
\]
First of all, notice that
\[
\norm{\varphi^{M_k}}_\beta^2
\geq 
\norm{Q(\varphi^{M_k})}_\beta^2 
=\sum_{n\notin\Lambda}\beta_n|\widehat{\varphi^{M_k}}(n)|^2
\geq c\sum_{n\notin\Lambda}|\widehat{\varphi^{M_k}}(n)|^2
=
c\norm{Q(\varphi^{M_k})}_2^2.
\]

Using the boundedness of the Paley projection in Theorem \ref{thm:Paley},
\[
\norm{Q(\varphi^{M_k})}_2^2 
=
\norm{\varphi^{M_k}}_2^2 - \norm{P(\varphi^{M_k})}_2^2
\geq  
\norm{\varphi^{M_k}}_2^2 - C_{\Lambda}^2\norm{\varphi^{M_k}}_{H^1}^2.
\]
On the one hand,
\begin{equation}\label{eq:normH2phiMk}  
\norm{\varphi^{M_k}}_2^2 
\geq 
\int_E\rho^{2M_k}\operatorname{d}\!m, 
\qquad k \in \N.  
\end{equation}

We focus on the quantity $\norm{\varphi^n}_{H^1}$, for $n\in\mathbb{N}$. By the Cauchy--Schwarz inequality,  
\[
\norm{\varphi^n}_{H^1} 
=
\int_E\rho^n\operatorname{d}\!m+\dfrac{1-m(E)}{2^n} 
\leq
\left(\int_E\rho^{2n}\operatorname{d}\!m\right)^\frac12 m(E)^{\frac12}+\dfrac{1-m(E)}{2^n}.
\]
Therefore,
\begin{equation}
\label{eq:H1-control}
C_{\Lambda}^2\norm{\varphi^n}^2_{H^1} 
\leq 2C_{\Lambda}^2m(E)\int_E\rho^{2n}\operatorname{d}\!m
+
C_{\Lambda}^2\dfrac{(1-m(E))^2}{2^{2n-1}}.
\end{equation}

If $n=M_k$ for some $k \geq S_\varepsilon$, by \eqref{eq:choice-2} we have that
\begin{equation}
\label{eq:observation-estimate}
C_{\Lambda}^2\dfrac{(1-m(E))^2}{2^{2M_k-1}}
\leq 
\dfrac{C_{\Lambda}^2}{2^{2M_k-1}} 
\leq 
\dfrac{\varepsilon}{16 M_k2^k} 
= 
\frac{m(E_k)}{16}. 
\end{equation}
On the other hand, we have
\begin{equation}
\label{eq:intE-estimate}
\int_E\rho^{2M_k}\operatorname{d}\!m 
\geq
\int_{E_k}\rho^{2M_k}\operatorname{d}\!m
=
\left(1-\dfrac{1}{M_k}\right)^{M_k}m(E_k) 
\geq 
\dfrac{m(E_k)}{4}.
\end{equation}
Combining \eqref{eq:H1-control}, \eqref{eq:observation-estimate} and \eqref{eq:intE-estimate},
\begin{equation}
C_{\Lambda}^2\norm{\varphi^n}^2_{H^1} 
\leq
(2m(E)C_{\Lambda}^2+\frac14) \int_E\rho^{2n}\operatorname{d}\!m.
\end{equation}
At this point we fix $\varepsilon$ such that
\[
2m(E)C_{\Lambda}^2=2C_{\Lambda}^2\varepsilon\sum_{k\in\N}\frac{1}{2^k M_k}\leq \frac14.
\]
Now, using also \eqref{eq:normH2phiMk}, we conclude that for every $k \geq S_\varepsilon$
\[
\norm{Q(\varphi^{M_k})}_2^2 
\geq  
\norm{\varphi^{M_k}}_2^2 - C_{\Lambda}^2\norm{\varphi^{M_k}}_{H^1}^2
\geq 
\dfrac{1}{2}\int_E\rho^{2M_k}\operatorname{d}\!m
\geq 
\frac{m(E_k)}{8},
\]
where in the last inequality we used \eqref{eq:intE-estimate}. We are in a position to conclude that $\varphi \not\in \mathrm{HS}(\mathbf{\beta})$, as desired. In fact,
\[
\sum_{k\geq S_\varepsilon}\frac{\|\varphi^{M_k}\|^2_\beta}{\beta_{M_k}}
\geq
\frac18 \sum_{k\geq S_\varepsilon} \frac{m(E_k)}{\beta_{M_k}}
=
\frac18\sum_{k \geq S_\varepsilon}\dfrac{1}{2^k\beta_{M_k}M_k}
\geq
\frac18\sum_{k \geq S_\varepsilon}2^k
=+\infty,
\]
where we have used \eqref{eq:choice-1} in the last inequality.
\end{proof}

\section{Further results}
\label{sec:further}

An even more ambitious version of the question posed in the Introduction is as follows.

\textbf{Question.} For a given operator-theoretic property, when does the class of composition operators on $H^2(\beta)$ with such a property coincide with the analogous class on $H^2(\beta')$?

Guided by this question, in this section we give sufficient conditions so that the inclusion $\mathrm{HS}(\beta) \subseteq \mathrm{HS}(\beta')$ holds. Our ideas are motivated by the proof of \cite[Theorem 3.3]{LLQRP}. To this extent, we first give a reformulation of \cite[Lemma 3.6]{LLQRP} (see also \cite{Katsnelson}).

\begin{lemma}
\label{lemma:HS-to-HS}
Let $A \colon \ell^2 \to \ell^2$ be a bounded operator represented by the matrix $(a_{m,n})_{m,n \in \N}$ with respect to the canonical basis $(e_n)_{n \in \N}$. Let $d = (d_n)_{n \in \N}$ be a sequence of positive numbers such that, for all $m,n \in \N$, it holds that $a_{m,n} = 0$ whenever $d_m > d_n$. Then,
$$\norm{DAD^{-1}}_{\mathcal{S}^2} \leq \norm{A}_{\mathcal{S}^2},$$
where $D$ is the (possibly unbounded) diagonal operator with entries given by $d$ (i.e., $Da = (d_na_n)_{n \in \N}$ for all $a = (a_n)_{n \in \N} \in \ell^2$).
\begin{proof}
Simply notice that
$$\left<DAD^{-1}e_n,e_m\right>_{\ell^2} = \dfrac{d_m}{d_n}a_{m,n}, \qquad n,m \in \N.$$
By hypothesis, if $d_m/d_n > 1$, then $\left<DAD^{-1}e_n,e_m\right> = 0$. Therefore, the result follows.
\end{proof}
\end{lemma}

Our main result in this section builds on the following definition, which is a generalization of the one given in \cite[Theorem 3.3]{LLQRP}.

\begin{definition}
Let $\beta = (\beta_n)_{n \in \N}$ and $\beta' = (\beta'_n)_{n \in \N}$ be weights. We say that $\beta'$ is weakly decreasing with respect to $\beta$ if the weight $(\beta'_n/\beta_n)_{n \in \N}$ is weakly decreasing. Namely, if for every $\delta > 0$ there exists a constant $C = C(\delta) > 0$ such that
$$\dfrac{\beta'_m}{\beta_m} \leq C\dfrac{\beta'_n}{\beta_n}, \qquad m > (1+\delta)n.$$
\end{definition}
Notice, in particular, that $\beta'$ is weakly decreasing with respect to $\mathbf{1}$ if and only if $\beta'$ is weakly decreasing in the sense of \cite{LLQRP}.
\begin{proposition}
\label{prop:rwd-implies-inclusion}
Let $\beta$ and $\beta'$ be weights. Assume that $\beta$ satisfies \eqref{eq:cond-D}, while
\begin{equation}
\label{eq:cond-exact-D}
\lim_{n \to \infty}\beta'^{1/n}_n = 1.
\end{equation}
If $\beta'$ is weakly decreasing with respect to $\beta$, then
$$\mathrm{HS}(\mathbf{\beta}) \cap \{\varphi \in \mathrm{Hol}(\D,\D) : \varphi(0) = 0\} \subseteq \mathrm{HS}(\beta').$$
\begin{proof}
Let $\varphi \in \mathrm{HS}(\mathbf{\beta})$ with $\varphi(0) = 0$. If $\abs{\varphi'(0)} = 1$ then $\varphi$ is an automorphism of the unit disk. Then, $C_{\varphi}$ is compact and bijective on $H^2(\beta)$. Since $H^2(\beta)$ is infinite dimensional, this is not possible. Therefore, we deduce that $\abs{\varphi'(0)} < 1$. By \cite[Lemma 3.4]{LLQRP}, there exists $\rho > 0$ so that
$$\abs{\widehat{\varphi^n}(m)} \leq e^{m-(1+\rho)n}, \qquad n,m \in \N.$$
In particular, taking $\delta = \rho/2$, we deduce that
\begin{equation}
\label{eq:bound-widehatvarphinm-2}
\abs{\widehat{\varphi^n}(m)} \leq e^{-\delta n}, \qquad n,m \in \N, \, m \leq (1+\delta)n.
\end{equation}

Now, let $\gamma = (\beta'_n/\beta_n)_{n \in \N}$. Since $\gamma$ is weakly decreasing, we can find $C = C(\delta) >~0$ so that $\gamma_m \leq C\gamma_n$ whenever $m \geq (1+\delta)n$. In fact, let $\widetilde{\gamma} = (\widetilde{\gamma}_n)_{n \in \N}$ be a new weight given by
$$\widetilde{\gamma}_n := \max(\{\gamma_n, \sup(\{\gamma_m : m > (1+\delta)n\})\}).$$
It is easy to check that:
\begin{enumerate}[\hspace{0.5cm}\rm(a)]
\item $\gamma$ and $\widetilde{\gamma}$ are equivalent, since $\gamma_n \leq \widetilde{\gamma}_n \leq C\gamma_n$ for all $n \in \N$.
\item $\widetilde{\gamma}$ is ``almost non-increasing'', since $\widetilde{\gamma}_m \leq \widetilde{\gamma}_n$ whenever $m \geq (1+\delta)n$, $n,m \in \N$.
\end{enumerate}

The result follows from these properties. In fact, by \eqref{eq:Hs-norm-comp-op}, we need to prove that
$$S := \sum_{n,m \in \N} \dfrac{\beta'_m}{\beta'_n}\abs{\widehat{\varphi^n}(m)}^2 < +\infty.$$
But, since $\varphi(0) = 0$, it is clear that $\widehat{\varphi^n}(m) = 0$ if $m < n$. Then,
$$S = \sum_{n = 1}^{\infty}\sum_{m = n}^{\infty} \dfrac{\beta'_m}{\beta'_n}\abs{\widehat{\varphi^n}(m)}^2.$$

On one hand, since $\beta$ and $\beta'$ satisfy \eqref{eq:cond-D} we see that there exists $N \in \N$ so that $e^{-xn} \leq \beta_n$ and $e^{-xn} \leq \beta'_n$ for all $n \geq N$, where $x = \delta/(\delta+2)$. Since $\beta'$ further satifies \eqref{eq:cond-exact-D}, we may also assume that $\beta'_n \leq e^{xn}$ for all $n \geq N$. Hence, we may find $C,c > 0$ so that $ce^{-xn} \leq \beta'_n$ and $ce^{-xn} \leq \beta'_n \leq Ce^{xn}$ for all $n \in \N$. Using these facts and \eqref{eq:bound-widehatvarphinm-2}, we have that
\begin{align}
\label{eq:is-it-avoidable}
S_1 := & \sum_{n = 1}^{\infty}\sum_{m = n}^{(1+\delta)n} \dfrac{\beta'_m}{\beta'_n}\abs{\widehat{\varphi^n}(m)}^2 \leq \dfrac{C}{c}\sum_{n = 1}^{\infty}\sum_{m = n}^{(1+\delta)n} \dfrac{e^{xm}}{e^{-xn}}e^{-2\delta n} \\
\leq & \dfrac{C}{c}\sum_{n = 1}^{\infty}(n\delta+1) e^{x(1+\delta)n}e^{xn}e^{-2\delta n}
=  \dfrac{C}{c}\sum_{n = 1}^{\infty}(n\delta+1) e^{-\delta n} < +\infty. \notag
\end{align}

It remains to show that
\begin{align*}
S_2 := & S - S_1 = \sum_{n = 1}^{\infty}\sum_{m = (1+\delta)n+1}^{\infty} \dfrac{\beta'_m}{\beta'_n}\abs{\widehat{\varphi^n}(m)}^2 \\
= & \sum_{n = 1}^{\infty}\sum_{m = (1+\delta)n+1}^{\infty} \dfrac{\gamma_m}{\gamma_n}\dfrac{\beta_m}{\beta_n}\abs{\widehat{\varphi^n}(m)}^2 < +\infty.
\end{align*}
By (a), this is equivalent to
$$\widetilde{S}_2 := \sum_{n = 1}^{\infty}\sum_{m = (1+\delta)n+1}^{\infty} \dfrac{\widetilde{\gamma_m}}{\widetilde{\gamma_n}}\dfrac{\beta_m}{\beta_n}\abs{\widehat{\varphi^n}(m)}^2 < +\infty.$$

But notice that $\widetilde{S}_2$ is the square of the Hilbert--Schmidt norm of some operator. In fact, $\widetilde{S}_2 = \|D\widetilde{A}D^{-1}\|^2_{\mathrm{HS}}$, where $A \colon \ell^2 \to \ell^2$ is the operator whose matrix representation $(a_{m,n})_{m,n \in \N}$ is given by
$$a_{m,n} = 
\begin{cases}
0, & \text{if } m \leq (1+\delta)n, \\
\sqrt{\dfrac{\beta_m}{\beta_n}}\widehat{\varphi^n}(m), & \text{otherwise},
\end{cases}
$$
and $D$ is the diagonal operator with entries $\widetilde{\gamma}$. By (b), if $\widetilde{\gamma}_m > \widetilde{\gamma}_n$, then $m < (1+\delta)n$, and so $a_{m,n} = 0$. Therefore, by Lemma \ref{lemma:HS-to-HS}, we have that
$$S_2 \leq \|\widetilde{A}\|^2_{\mathrm{HS}} = \sum_{n = 1}^{\infty}\sum_{m = (1+\delta)n+1}^{\infty}\dfrac{\beta_m}{\beta_n}\abs{\widehat{\varphi^n}(m)}^2 \leq \|C_{\varphi}\|^2_{\mathrm{HS}} < +\infty.$$ 
\end{proof}
\end{proposition}

Recall that every composition operator can be decomposed as $C_{\varphi} = C_{\varphi_0} \circ C_T$, where $T$ is an automorphism of $\D$, $\varphi_0(0) = 0$, and $\varphi = T \circ \varphi_0$. If all composition operators associated to automorphisms are bounded, since the class of Hilbert--Schmidt operators is an ideal (cf. \cite[p. 27]{Zhu}), we have that $C_{\varphi}$ is a Hilbert--Schmidt operator if and only if $C_{\varphi_0}$ is. Using this reasoning, the combination of \cite[Theorem 1.1]{LLQRP} and Proposition \ref{prop:rwd-implies-inclusion} yields the following results.

\begin{corollary}
Let $\beta$ and $\beta'$ be slowly oscillating weights. Assume that $\beta$ satisfies \eqref{eq:cond-D}, while
$$\lim_{n \to \infty}\beta'^{1/n}_n = 1.$$
If $\beta'$ is weakly decreasing with respect to $\beta$, then $\mathrm{HS}(\mathbf{\beta}) \subseteq \mathrm{HS}(\beta')$.
\end{corollary}

\section{Final remarks}
\label{sec:questions}
Despite our results, the questions (\textbf{Q1}) and (\textbf{Q2}) stated in the Introduction remain open in the general case. We would like to propose some further questions in a similar spirit.

The first of them was indicated in the Introduction. The other correspond to generalizations of (\textbf{Q1}) and (\textbf{Q2}), and also to Corollary \ref{cor:h2}.

\begin{itemize}
    
\item[(\textbf{Q3})] Does there exist a weight $\beta$ satisfying \eqref{eq:zero-liminf-limsup-infty} such that $\operatorname{K}(\beta) = \operatorname{K}(\mathbf{1})$?

\item[(\textbf{Q4})] Let $\beta$ and $\beta'$ be weights. When does $\operatorname{B}(\beta) = \operatorname{B}(\beta')$? When does $\operatorname{K}(\beta) = \operatorname{K}(\beta')$? For $1 \leq p < \infty$, when does $\operatorname{S}_p(\beta) = \operatorname{S}_p(\beta')$?

\item[(\textbf{Q5})] Let $\beta$ and $\beta'$ be weights. Assume that $\operatorname{B}(\beta) = \operatorname{B}(\beta')$ and $\operatorname{K}(\beta) = \operatorname{K}(\beta')$. Is it necessary that $\beta \simeq \beta'$?
\end{itemize}

\textbf{Statement on the use of AI tools.}
The reader may have noticed that the proofs of many important results of this paper are done by arguing by contradiction, and constructing an exotic symbol $\varphi$ that has or does not have a specific property. This is for example the case of Proposition \ref{P:betabounded}, Theorem \ref{T:lacunaryKsec4} or the necessity part of Proposition \ref{P:secondinclusionHSlacunary}. The downside of this technique is that, sometimes, one single counterexample does not properly illustrate the reason behind the phenomenon. In general, a direct (possibly constructive) proof is more enlightening.

However, this is a very natural approach to the theory of composition operators on spaces of analytic functions, due to the great diversity of symbols and spaces: several papers in this field are mostly example-based (e.g., \cite{LLQRP-Examples}). Also, in some instances, it can produce stronger statements in a rather easy way. Take for example the case of Proposition~\ref{P:betabounded}. A similar result for the case of boundedness was already established (namely, that if $B(\beta)=B(\mathbf{1})$ then $\beta$ is bounded) with a direct proof. Arguing by contradiction, assuming that $\beta$ is unbounded, we found an example $\phi$ that not only belongs to $B(\mathbf{1})\setminus B(\beta)$, but also to the smaller set $S_1(\mathbf{1})\setminus B(\beta)$, therefore giving a stronger result. 

Finally, this approach of looking for counterexamples fits very naturally in this new era of mathematics, in which the researchers also have at disposal LLMs, who possess striking computational capabilities and have access to an incredibly vast quantity of information and scientific literature.

In our case, after establishing the main result, Theorem \ref{T:main1}, we wanted to see whether it would be possible to improve our ideas and eliminate the current convergence hypothesis, thus solving (\textbf{Q1}) and (\textbf{Q2}). To do so, we used ChatGPT 6-Sol to find a plausible example of a lacunary weight $\beta$ satisfying \eqref{eq:zero-liminf-limsup-infty} for which $\operatorname{HS}(\beta) = \operatorname{HS}(\mathbf{1})$. ChatGPT proposed such an example making use of Paley's result \cite{Paley}, thus confirming that the convergence hypothesis in Theorem~\ref{T:main1} is crucial. Afterwards, Theorem \ref{thm:counterexample} and the content of Section \ref{sec:proof2} were obtained pushing the ideas behind the construction of this example.

To conclude, all mathematical statements, proofs, and references were independently checked, revised, and finalized by the authors, who take full responsibility for the content of the paper. No part of this manuscript has been autonomously written by any AI tool.

\end{document}